\documentclass{article}

\usepackage[a4paper]{geometry}
\usepackage{amsmath}
\usepackage{amssymb}
\usepackage{amsthm}
\usepackage{caption2}
\usepackage{multirow}
\usepackage{enumitem}

\newtheorem{thm}{Theorem}[section]
\newtheorem{rem}[thm]{Remark}

\newtheorem{lem}[thm]{Lemma}
\newtheorem{prop}[thm]{Proposition}

\newtheorem{claim}[thm]{Claim}

\title{Multi-part cross-intersecting families}

\author{Xiangliang Kong$^{\text{a}}$, Yuanxiao Xi$^{\text{b}}$ and Gennian Ge$^{\text{a,}}$\thanks{
  The research of G. Ge was supported by  the National Natural Science Foundation of China under Grant Nos.11431003 and 61571310, Beijing Scholars Program, Beijing Hundreds of Leading Talents Training Project of Science and Technology, and Beijing Municipal Natural Science Foundation.}\\
  \footnotesize $^{\text{a}}$ School of Mathematical Sciences, Capital Normal University, Beijing, 100048, China\\
  \footnotesize $^{\text{b}}$ School of Mathematical Sciences, Zhejiang University, Hangzhou 310027, Zhejiang, China}

\begin{document}
\date{}
\maketitle

\begin{abstract}
  Let $\mathcal{A}\subseteq{[n]\choose a}$ and $\mathcal{B}\subseteq{[n]\choose b}$ be two families of subsets of $[n]$, we say $\mathcal{A}$ and $\mathcal{B}$ are cross-intersecting if $A\cap B\neq \emptyset$ for all $A\in\mathcal{A}$, $B\in\mathcal{B}$. In this paper, we study cross-intersecting families in the multi-part setting. By characterizing the independent sets of vertex-transitive graphs and their direct products, we determine the sizes and structures of maximum-sized multi-part cross-intersecting families. This generalizes the results of Hilton's~\cite{H77} and Frankl--Tohushige's~\cite{FT92} on cross-intersecting families in the single-part setting.


\smallskip


\end{abstract}

\section{Introduction}

Let $[n]$ be the standard $n$-element set. For an integer $0\leq k\leq n$, denote ${[n]\choose k}$ as the family of all $k$-element subsets of $[n]$. A family $\mathcal{F}$ is said to be $intersecting$ if $A\cap B\neq \emptyset$, for any $A,B\in\mathcal{F}$. The celebrated Erd\H{o}s--Ko--Rado theorem~\cite{EKR} says that if $\mathcal{F}\subseteq {[n]\choose k}$ is an intersecting family with $1\leq k\leq \frac{n}{2}$, then
\begin{equation*}
|\mathcal{F}|\leq {{n-1}\choose {k-1}},
\end{equation*}
and if $n>2k$, the equality holds if and only if every subset in $\mathcal{F}$ contains a fixed element.

Because of its fundamental status in extremal set theory, the Erd\H{o}s--Ko--Rado theorem has numerous extensions in different ways. One of the directions is to study $cross$-$t$-$intersecting$ families: Denote $2^{[n]}$ as the $power~set$ of $[n]$, let $\mathcal{A}_i\subseteq 2^{[n]}$ for each $1\leq i\leq m$, $\mathcal{A}_1,\mathcal{A}_2,\ldots,\mathcal{A}_m$ are said to be cross-t-intersecting, if $|A\cap B|\geq t$ for any $A\in\mathcal{A}_i$ and $B\in\mathcal{A}_j$, $i\neq j$. Especially, we say $\mathcal{A}_1,\mathcal{A}_2,\ldots,\mathcal{A}_m$ are cross-intersecting if $t=1$. 

Hilton~\cite{H77} investigated the cross-intersecting families in ${[n]\choose k}$, and proved the following inequality:
\begin{thm}\emph{(\cite{H77})}\label{h77}
Let $\mathcal{A}_1,\mathcal{A}_2,\ldots, \mathcal{A}_m$ be cross-intersecting families in ${[n]\choose k}$ with $n\geq 2k$. Then
\begin{flalign*}
\sum\limits_{i=1}^{m}|\mathcal{A}_i|\leq\left\{\begin{array}{ll}{n\choose k},&~~~\text{if}~m\leq\frac{n}{k};\\
m\cdot{{n-1}\choose {k-1}},&~~~\text{if}~m\geq\frac{n}{k}.\end{array}\right.
\end{flalign*}
\end{thm}

In the same paper, Hilton also determined the structures of $\mathcal{A}_i$'s when the equality holds. Since then, there have been many extensions about Theorem~\ref{h77}. Borg~\cite{B09} gave a simple proof of Theorem~\ref{h77}, and generalized it to labeled sets~\cite{B08}, signed sets~\cite{BL10} and permutations~\cite{B10}. Using the results of the independent number about vertex-transitive graphs, Wang and Zhang~\cite{WZ11} extended this theorem to general symmetric systems, which comprise finite sets, finite vector spaces and permutations, etc.

Hilton and Milner~\cite{HM67} also dealt with pairs of cross-intersecting families in ${[n]\choose k}$ when neither of the two families is empty:
\begin{thm}\emph{(\cite{HM67})}\label{hm67}
Let $\mathcal{A},\mathcal{B}\subseteq {[n]\choose k}$ be non-empty cross-intersecting families with $n\geq 2k$. Then $|\mathcal{A}|+|\mathcal{B}|\leq {n\choose k}-{{n-k}\choose k}+1$.
\end{thm}

This result was generalized by Frankl and Tokushige~\cite{FT92} to the case when $\mathcal{A}$ and $\mathcal{B}$ are not necessarily in the same $k$-uniform subfamily of $2^{[n]}$:
\begin{thm}\emph{(\cite{FT92})}\label{FT92}
Let $\mathcal{A}\subseteq {[n]\choose a}$ and $\mathcal{B}\subseteq {[n]\choose b}$ be non-empty cross-intersecting families with $n\geq a+b$, $a\leq b$. Then $|\mathcal{A}|+|\mathcal{B}|\leq {n\choose b}-{{n-a}\choose b}+1$.
\end{thm}

In~\cite{WZ13}, Wang and Zhang generalized Theorem~\ref{FT92} to cross-t-intersecting families. Recently, using shifting techniques, Frankl and Kupavskii~\cite{FK17} gave another proof of the result of Wang and Zhang for the case when $\mathcal{A},\mathcal{B}\subseteq{[n]\choose k}$.

As another direction, the multi-part extension of the Erd\H{o}s--Ko--Rado problem was introduced by Frankl~\cite{F96}, in connection with a similar result of Sali~\cite{Sali92}. For an integer $p\geq 1$ and positive integers $n_1,\ldots, n_p$, take $[\sum_{i\in[p]}n_i]$ to be the ground set. Then this ground set can be viewed as the disjoint union of $p$ parts $\bigsqcup_{i\in[p]}S_i$, where $S_1=[n_1]$ and $S_i=\{1+\sum_{j\in[i-1]}n_j,\ldots,\sum_{j\in[i]}n_j\}$ for $2\leq i\leq p$. More generally, denote $2^{S_i}$ as the $power~ set$ of $S_i$, for sets $A_1\in 2^{S_1},\ldots, A_p\in 2^{S_p}$, let $\bigsqcup_{i\in[p]} A_i$ be the subset of $\bigsqcup_{i\in[p]}S_i$ with $A_i$ in the $i$-th part, and for families $\mathcal{F}_1\subseteq 2^{S_1},\ldots,\mathcal{F}_p\subseteq 2^{S_p}$, let $\prod_{i\in[p]}\mathcal{F}_i=\{\bigsqcup_{i\in[p]} A_i:~A_i\in \mathcal{F}_i\}$. Then consider $k_1\in[n_1],\ldots, k_p\in[n_p]$, we denote $\prod_{i\in[p]}{[n_i]\choose k_i}$ as the family of all subsets of $\bigsqcup_{i\in[p]}S_i$ which have exactly $k_i$ elements in the $i$-th part. Therefore, families of the form $\mathcal{F}\subseteq \prod_{i\in [p]}{[n_i]\choose k_i}$ can be viewed as the natural generalization of $k$-uniform families to the multi-part setting. Similarly, a multi-part family is intersecting if any two sets of this family intersect in at least one of the $p$ parts.

Frankl proved that for any integer $p\geq 1$, any positive integers $n_1,\ldots,n_p$ and $k_1,\ldots,k_p$ satisfying $\frac{k_1}{n_1}\leq\ldots\leq\frac{k_p}{n_p}\leq\frac{1}{2}$, if $\mathcal{F}\subseteq\prod_{i\in[p]}{[n_i]\choose k_i}$ is a multi-part intersecting family, then
\begin{equation*}
|\mathcal{F}|\leq\frac{k_p}{n_p}\cdot\prod_{i\in[p]}{n_i\choose k_i}.
\end{equation*}
This bound is sharp, for example, it is attained by the following family:
\begin{equation*}
\{A\in{[n_p]\choose k_p}:~i\in A,\text{~for~some~}i\in[n_p]\}\times\prod_{i\in[p-1]}{[n_i]\choose k_i}.
\end{equation*}

Recall that the $Kneser~graph$ $KG_{n,k}$ is the graph on the vertex set ${[n]\choose k}$, with $A,B\in{[n]\choose k}$ forming an edge if and only if $A\cap B=\emptyset$. And an intersecting subfamily of ${[n]\choose k}$ corresponds to an independent set in $KG_{n,k}$. Hence an intersecting subfamily of $\prod_{i\in[p]}{[n_i]\choose k_i}$ corresponds to an independent set in the graph (direct) product $KG_{n_1,k_1}\times\ldots\times KG_{n_p,k_p}$. Therefore, Frankl's result can be viewed as a consequence of the general fact that $\alpha(G\times H)=\max\{\alpha(G)|H|,\alpha(H)|G|\}$ for vertex-transitive graphs $G$ and $H$, which was proved by Zhang in~\cite{zhang2012independent}.

Recently, Kwan, Sudakov and Vieira~\cite{KSV18} considered a stability version of the Erd\H{o}s--Ko--Rado theorem in the multi-part setting. They determined the maximum size of the non-trivially intersecting multi-part family when all the $n_i$'s are sufficiently large. This disproved a conjecture proposed by Alon and Katona, which was also mentioned in~\cite{Katona2017}.

In this paper, we extend Theorem~\ref{h77} and Theorem~\ref{FT92} to the $multi$-$part$ version. For any subset $S\subseteq[n]$, denote $\bar{S}$ as the complementary set of $S$ in $[n]$. Moreover, let $\mathcal{F}$ be a family of subsets of $[n]$, for any subfamily $\mathcal{A}\subseteq\mathcal{F}$, denote ${\mathcal{A}}_{\mathcal{F}}=\{B\in\mathcal{F}:A\cap B=\emptyset\text{ for some }A\in\mathcal{A}\}$. 

Our main results are as follows.

\begin{thm}\label{main0}
Given a positive integer $p$, let $n_1,n_2,\ldots,n_p$ and $k_1,k_2,\ldots,k_p$ be positive integers satisfying $n_i\geq 2k_i$ for all $i\in[p]$. Let $X=\prod_{i\in[p]}{[n_i]\choose k_i}$ and $\mathcal{A}_1,\mathcal{A}_2,\ldots,\mathcal{A}_m$ be cross-intersecting families over $X$ with $\mathcal{A}_1\neq \emptyset$. Then
\begin{flalign}\label{bound1}
\sum\limits_{i=1}^{m}|\mathcal{A}_i|\leq\left\{\begin{array}{ll}|X|,&~~~\text{if}~m\leq\min\limits_{i\in[p]}\frac{n_i}{k_i};\\m\cdot M,&~~~\text{if}~m\ge\min\limits_{i\in[p]}\frac{n_i}{k_i},\end{array}\right.
\end{flalign}
where $M=\max_{i\in[p]}{n_i-1\choose k_i-1}\prod_{j\ne i}{n_j\choose k_j}$. Furthermore, the bound is attained if and only if one of the following holds:
\begin{itemize}
  {\item[(i)] $m<\min_{i\in[p]}\frac{n_i}{k_i}$ and $\mathcal{A}_1=X$, $\mathcal{A}_2=\cdots=\mathcal{A}_m=\emptyset$;}
  {\item[(ii)] $m>\min_{i\in[p]}\frac{n_i}{k_i}$ and $\mathcal{A}_1=\cdots=\mathcal{A}_m=I$, where $I$ is a maximum intersecting family in $X$;}
  {\item[(iii)] $m=\min_{i\in[p]}\frac{n_i}{k_i}$ and $\mathcal{A}_1,\ldots,\mathcal{A}_m$ are as in $(i)$ or $(ii)$, or there exists a non-empty set $S_1\subseteq\{s\in[p]:~\frac{n_s}{k_s}=2\}$ and $\mathcal{F}=\prod_{s\in S_1}{[n_s]\choose k_s}$ such that
   \begin{flalign}\label{formula}
   \begin{array}{l}
   \mathcal{A}_1=(\mathcal{A}\sqcup(\mathcal{E}\cup \mathcal{E}_{\mathcal{F}}))\times\prod\limits_{s\in [p]\setminus S_1}{[n_s]\choose k_s}~\text{and}~
   \mathcal{A}_2=(\mathcal{A}\sqcup(\mathcal{E}'\cup\mathcal{E}'_{\mathcal{F}}))\times\prod\limits_{s\in [p]\setminus S_1}{[n_s]\choose k_s}
   \end{array}
   \end{flalign}
   for some $\mathcal{A}$, $\mathcal{E}$, $\mathcal{E}'\subseteq\mathcal{F}$, where $\mathcal{A}=\{A_1,\ldots,A_{w_0}\}$ satisfying $2w_0< |\mathcal{F}|$ and $A_i\ne \bar{A}_j$ for all $i\ne j\in [w_0]$, $\mathcal{E}\sqcup\mathcal{E}'=\{E_1,\ldots,E_v\}$ and $\mathcal{E}_\mathcal{F}\sqcup\mathcal{E}'_\mathcal{F}=\{\bar{E}_1,\ldots,\bar{E}_v\}$ satisfying $2(v+w_0)=\prod_{s\in S_1}{n_s\choose k_s}$ and $\sqcup_{j=1}^v\{E_j,\bar{E}_j\}=\mathcal{F}\setminus(\mathcal{A}\sqcup \mathcal{A}_\mathcal{F})$.
}
\end{itemize}
\end{thm}

\begin{rem}\label{rem1}
In \cite{WZ11}, the authors proved a similar result (Theorem 2.5 in \cite{WZ11}) for general connected symmetric systems. Actually, Theorem \ref{main0} can be viewed as an application of the method involved to obtain Theorem 2.5 in \cite{WZ11}. But different from the general case, Theorem \ref{main0} determines all the exact structures when the bound in (\ref{bound1}) is attained.
\end{rem}

\begin{thm}\label{main1}
For any $p\geq 2$, let $n_i, t_i, s_i$ be positive integers satisfying $n_i\geq s_i+t_i+1$, $2\leq s_i,t_i\leq \frac{n_i}{2}$ for every $i\in[p]$ and $n_i\leq\frac{7}{4}n_j$ for all distinct $i,j\in [p]$. If $\prod_{i\in [p]}{n_i\choose s_i}\geq \prod_{i\in [p]}{n_i\choose t_i}$ and $\mathcal{A}\subseteq\prod_{i\in [p]}{[n_i]\choose t_i},~\mathcal{B}\subseteq\prod_{i\in [p]}{[n_i]\choose s_i}$ are non-empty cross-intersecting families, then
\begin{equation}\label{eq02}
|\mathcal{A}|+|\mathcal{B}|\leq \prod_{i\in [p]}{n_i\choose s_i}-\prod_{i\in [p]}{{n_i-t_i}\choose s_i}+1,
\end{equation}
and the bound is attained if and only if the following holds:
\begin{itemize}
  {\item[(i)] $\prod_{i\in [p]}{n_i\choose s_i}\geq \prod_{i\in [p]}{n_i\choose t_i}$, $\mathcal{A}=\{A\}$ and $\mathcal{B}=\{B\in \prod_{i\in [p]}{[n_i]\choose s_i}:B\cap A\neq\emptyset\}$ for some $A\in \prod_{i\in [p]}{[n_i]\choose t_i}$;}
  {\item[(ii)] $\prod_{i\in [p]}{n_i\choose s_i}= \prod_{i\in [p]}{n_i\choose t_i}$, $\mathcal{B}=\{B\}$ and $\mathcal{A}=\{A\in \prod_{i\in [p]}{[n_i]\choose t_i}:B\cap A\neq \emptyset\}$ for some $B\in \prod_{i\in [p]}{[n_i]\choose s_i}$.}
\end{itemize}
\end{thm}

\begin{rem}\label{condition_of_main1}
The restrictions $s_i, t_i\leq \frac{n_i}{2}$  for every $i\in[p]$ and $n_i\leq\frac{7}{4}n_j$ for all distinct $i,j\in [p]$ in Theorem~\ref{main1} are necessary.

When $s_i, t_i\leq \frac{n_i}{2}$ is violated, for example, taking $n_1=n_2=18$, $(s_1,t_1)=(15,2)$ and $(s_2,t_2)=(2,3)$, we have $s_1>\frac{n_1}{2}$ and ${n_1\choose s_1}\cdot{n_2\choose s_2}={n_1\choose t_1}\cdot{n_2\choose t_2}$. Set $\mathcal{A}=\{A_1\}\times{[n_2]\choose t_2}$ for some 2-subset $A_1\subseteq [n_1]$ and $\mathcal{B}=\{B_1\in {[n_1]\choose s_1}:B_1\cap A_1\neq\emptyset\}\times{[n_2]\choose s_2}$. Then we have $|\mathcal{A}|+|\mathcal{B}|>{n_1\choose s_1}\cdot{n_2\choose s_2}-{{n_1-t_1}\choose s_1}\cdot{{n_2-t_2}\choose s_2}+1$. 

As for the restriction $n_i\leq\frac{7}{4}n_j$, the constant $\frac{7}{4}$ here might not be tight, but the quantities of $n_i,n_j$ for distinct $i,j\in[p]$ need to be very close. For example, taking $n_1=5,n_2=12$ and $(s_1,t_1)=(s_2,t_2)=(2,2)$, we have $n_2>\frac{7}{4}n_1$ and ${n_1\choose s_1}\cdot{n_2\choose s_2}={n_1\choose t_1}\cdot{n_2\choose t_2}$. Similarly, set $\mathcal{A}=\{A_1\}\times{[n_2]\choose t_2}$ for some 2-subset $A_1\subseteq [n_1]$ and $\mathcal{B}=\{B_1\in {[n_1]\choose s_1}:B_1\cap A_1\neq\emptyset\}\times{[n_2]\choose s_2}$. Then we have $|\mathcal{A}|+|\mathcal{B}|>{n_1\choose s_1}\cdot{n_2\choose s_2}-{{n_1-t_1}\choose s_1}\cdot{{n_2-t_2}\choose s_2}+1$.

The families $\mathcal{A}$ and $\mathcal{B}$ we constructed here are closely related to imprimitive subsets of ${[n_1]\choose t_1}\times{[n_2]\choose t_2}$, which we will discuss later in Section 2.2 and Section 4.
\end{rem}

We shall introduce some results about the independent sets of vertex-transitive graphs and their direct products in the next section, and prove Theorem~\ref{main0} in Section 3, Theorem~\ref{main1} in Section 4. In Section 5, we will conclude the paper and discuss some remaining problems. For the convenience of the proof, if there is no confusion, we will denote $\prod_{i\in[p]}A_i$ as the subset $\bigsqcup_{i\in[p]}A_i\subseteq \bigsqcup_{i\in[p]}S_i$ in the rest of the paper.

\section{Preliminary results}

\subsection{Independent sets of vertex-transitive graphs}

Given a finite set $X$, for every $A\subseteq X$, denote $\bar{A}=X\setminus A$. For a simple graph $G=G(V,E)$, denote $\alpha(G)$ as the independent number of $G$ and $I(G)$ as the set of all maximum independent sets of $G$. For $v\in V(G)$, define the neighborhood $N_G(v)=\{u\in V(G):(u,v)\in E(G)\}$. For a subset $A\subseteq V(G)$, write $N_G(A)=\{b\in V(G):(a,b)\in E(G)$ for some $a\in A\}$ and $N_G[A]=A\cup N_G(A)$, if there is no confusion, we denote them as $N(A)$ and $N[A]$ for short respectively.

A graph $G$ is said to be vertex-transitive if its automorphism group $\Gamma(G)$ acts transitively upon its vertices. As a corollary of the ``No-Homomorphism'' lemma for vertex-transitive graphs in~\cite{albertson1985homomorphisms}, Cameron and Ku~\cite{CK03} proved the following theorem.

\begin{thm}\label{CK}\emph{(\cite{CK03})}
Let $G$ be a vertex-transitive graph and $B$ a subset of $V(G)$. Then any independent set $S$ in $G$ satisfies that $\frac{|S|}{|V(G)|}\leq\frac{\alpha(G[B])}{|B|}$, equality implies that $|S\cap B|=\alpha(G[B])$.
\end{thm}

Using the above theorem, Zhang~\cite{zhang2011primitivity} proved the following result.
\begin{lem}\emph{(\cite{zhang2011primitivity})}\label{subgraph independent set}
Let $G$ be a vertex-transitive graph, and $A$ be an independent set of $G$, then $\frac{|A|}{|N_G[A]|}\leq\frac{\alpha(G)}{|G|}$. Equality implies that $|S\cap N_G[A]|=|A|$ for every $S\in I(G)$, and in particular $A\subseteq S$ for some $S\in I(G)$.
\end{lem}

An independent set $A$ in $G$ is said to be imprimitive if $|A|<\alpha(G)$ and $\frac{|A|}{|N[A]|}=\frac{\alpha(G)}{|G|}$, and $G$ is called IS-imprimitive if $G$ has an imprimitive independent set. Otherwise, $G$ is called IS-primitive. Note that a disconnected vertex-transitive graph $G$ is always IS-imprimitive. Hence IS-primitive vertex-transitive graphs are all connected.

The following inequality about the size of an independent set $A$ and its non-neighbors $\bar{N}[A]$ is crucial for the proof of Theorem~\ref{main0}.

\begin{lem}\label{lemmaimprimitive} Let $G$ be a vertex transitive graph, and let $A$ be an independent set of $G$. Then
\begin{flalign}
|A|+\frac{\alpha(G)}{|G|}|\bar{N}[A]|\leq\alpha(G).\nonumber
\end{flalign}
Equality holds if and only if $A=\emptyset$ or $|A|=\alpha(G)$ or $A$ is an imprimitive independent set.
\end{lem}
For the integrity of the paper, we include the proof here. In~\cite{WZ11}, Wang and Zhang proved the same inequality for a more generalized combinatorial structure called $symmetric$ $system$ (see~\cite{WZ11}, Corollary 2.4).

\begin{proof}
If $A=\emptyset$ or $A=\alpha(G)$, the equality trivially holds. Suppose $0<|A|<\alpha(G)$, and let $B$ be a maximal independent set in $\bar{N}[A]$, then $|B|=\alpha(\bar{N}(A))$. Clearly,  $A\cup B$ is also an independent set of $G$, thus we have $|A|+|B|\leq\alpha(G)$. By Theorem $\ref{CK}$, we obtain that $\frac{|B|}{|\bar{N}[A]|}\ge\frac{\alpha(G)}{|G|}$. Therefore,
\begin{flalign}
|A|+\frac{\alpha(G)}{|G|}|\bar{N}[A]|\leq|A|+|B|\leq\alpha(G),\nonumber
\end{flalign}
the equality holds when $\alpha(G)=|A|+\frac{\alpha(G)}{|G|}|\bar{N}[A]|=|A|+\frac{\alpha(G)}{|G|}(|G|-|N[A]|)$, which leads to $\frac{|A|}{|N[A]|}=\frac{\alpha(G)}{|G|}$, i.e., $A$ is an imprimitive independent set.
\end{proof}

Let $X$ be a finite set, and $\Gamma$ a group acting transitively on $X$. Then $\Gamma$ is said to be primitive on $X$ if it preserves no nontrivial partition of $X$. A vertex-transitive graph $G$ is called primitive if the automorphism $\text{Aut}(G)$ is primitive on $V(G)$. To show the connection between the primitivity and the IS-primitivity of a vertex-transitive graph $G$, Zhang (see Proposition 2.4 in \cite{zhang2011primitivity}) proved that if $G$ is primitive, then it must be IS-primitive. As a consequence of this result, Wang and Zhang~\cite{WZ11} derived the IS-primitivity of the Kneser graph.



\begin{prop}\emph{(\cite{WZ11})}\label{imprimitive Kneser graph}
The Kneser graph $KG_{n,k}$ is IS-primitive except for $n=2k\ge 4$.
\end{prop}

In order to deal with the multi-part case, we also need the results about the independent sets in direct products of vertex-transitive graphs. Let $G$ and $H$ be two graphs, the direct product $G\times H$ of $G$ and $H$ is defined by
\begin{flalign*}
V(G\times H)=V(G)\times V(H),
\end{flalign*}
and
\begin{flalign*}
E(G\times H)=\{[(u_1,v_1),(u_2,v_2)]:(u_1,u_2)\in E(G)\text{ and }(v_1,v_2)\in E(H)\}.
\end{flalign*}
Clearly, $G\times H$ is a graph with $\text{Aut}(G)\times\text{Aut}(H)$ as its automorphism group. And, if $G,H$ are vertex-transitive, then $G\times H$ is also vertex-transitive under the actions of $\text{Aut}(G)\times\text{Aut}(H)$. We say the direct product $G\times H$ is MIS-normal (maximum-independent-set-normal) if every maximum independent set of $G\times H$ is a preimage of an independent set of one factor under projections.

In \cite{zhang2012independent}, Zhang obtained the exact structure of the maximal independent set of $G\times H$.

\begin{thm}\emph{(\cite{zhang2012independent})}\label{tensor product} Let $G$ and $H$ be two vertex-transitive graphs with $\frac{\alpha(G)}{|G|}\ge\frac{\alpha(H)}{|H|}$. Then
\begin{flalign}
\alpha(G\times H)=\alpha(G)|H|,\nonumber
\end{flalign}
and exactly one of the following holds:
\begin{itemize}
\item[(i)] $G\times H$ is MIS-normal;
\item[(ii)] $\frac{\alpha(G)}{|G|}=\frac{\alpha(H)}{|H|}$ and one of $G$ or $H$ is IS-imprimitive;
\item[(iii)] $\frac{\alpha(G)}{|G|}>\frac{\alpha(H)}{|H|}$ and $H$ is disconnected.
\end{itemize}
\end{thm}

In fact if $\frac{\alpha(G)}{|G|}=\frac{\alpha(H)}{|H|}$  and $A$ is an imprimitive independent set of $G$, then for every $I\in I(H)$, it is easy to see that $S=(A\times V(H))\cup (\bar{N}[A]\times I)$ is an independent set of $G\times H$ with size $\alpha(G)|H|$.

Zhang \cite{zhang2011primitivity} also investigated the relationship between the graph primitivity and the structures of the maximum independent sets in direct products of vertex-transitive graphs.
\begin{thm}\emph{(\cite{zhang2011primitivity})}\label{normal-imprimitive} Suppose $G\times H$ is MIS-normal and $\frac{\alpha(H)}{|H|}\leq\frac{\alpha(G)}{|G|}$. If $G\times H$ is IS-imprimitive, then one of the following two possible cases holds:
 \begin{itemize}
\item[(i)] $\frac{\alpha(G)}{|G|}=\frac{\alpha(H)}{|H|}$ and one of them is IS-imprimitive or both $G$ and $H$ are bipartite;
\item[(ii)] $\frac{\alpha(G)}{|G|}>\frac{\alpha(H)}{|H|}$ and $G$ is IS-imprimitive.
 \end{itemize}
\end{thm}

As an application of Theorem~\ref{tensor product} and Theorem~\ref{normal-imprimitive}, Geng et al.~\cite{geng2012structure} showed the MIS-normality of the direct products of Kneser graphs.

\begin{thm}\emph{(\cite{geng2012structure})}\label{imprimitive product kneser graph} Given a positive integer $p$, let $n_1,n_2,\ldots,n_p$ and $k_1,k_2,\ldots,k_p$ be $2p$ positive integers with $n_i\ge 2k_i$ for $1\leq i\leq p$. Then the direct product of the Kneser graph $$\text{KG}_{n_1,k_1}\times\text{KG}_{n_2,k_2}\times\cdots\times\text{KG}_{n_p,k_p}$$ is MIS-normal except that there exist $i,j$ and $\ell$ with $n_i=2k_i\ge 4$ and $n_j=2k_j$, or $n_i=n_j=n_{\ell}=2$.
\end{thm}

\subsection{Nontrivial independent sets of part-transitive bipartite graphs}

For a bipartite graph $G(X,Y)$ with two parts $X$ and $Y$, an independent set $A$ is said to be non-trivial if $A\nsubseteq X$ and $A\nsubseteq Y$. $G(X,Y)$ is said to be part-transitive if there is a group $\Gamma$ acting transitively upon each part and preserving its adjacency relations. Clearly, if $G(X,Y)$ is part-transitive, then every vertex of $X~(Y)$ has the same degree, written as $d(X)~(d(Y))$. We use $\alpha(X,Y)$ and $I(X,Y)$ to denote the size and the set of the maximum-sized nontrivial independent sets of $G(X,Y)$, respectively.

Let $G(X,Y)$ be a non-complete bipartite graph and let $A\cup B$ be a nontrivial independent set of $G(X,Y)$, where $A\subseteq X$ and $B\subseteq Y$. Then $A\subseteq X\setminus N(B)$ and $B\subseteq Y\setminus N(A)$, which implies
\begin{equation*}
|A|+|B|\leq \max{\{|A|+|Y|-|N(A)|, |B|+|X|-|N(B)|\}}.
\end{equation*}
So we have
\begin{equation}\label{eq03}
\alpha(X,Y)=\max{\{|Y|-\epsilon(X), |X|-\epsilon(Y)\}},
\end{equation}
where $\epsilon(X)=\min\{|N(A)|-|A|: A\subseteq X, N(A)\neq Y\}$ and $\epsilon(Y)=\min\{|N(B)|-|B|: B\subseteq Y, N(B)\neq X\}$.

We call $A\subseteq X$ a fragment of $G(X,Y)$ in $X$ if $N(A)\neq Y$ and $|N(A)|-|A|=\epsilon(X)$, and we denote $\mathcal{F}(X)$ as the set of all fragments in $X$. Similarly, we can define $\mathcal{F}(Y)$. Furthermore, denoting $\mathcal{F}(X,Y)=\mathcal{F}(X)\cup\mathcal{F}(Y)$, we call an element $A\in \mathcal{F}(X,Y)$ a $k$-fragment if $|A|=k$. And we call a fragment $A\in\mathcal{F}(X)$ trivial if $|A|=1$ or $A=X\setminus N(b)$ for some $b\in Y$. Since for each $A\in \mathcal{F}(X)$, $Y\setminus N(A)$ is a fragment in $\mathcal{F}(Y)$. Hence, once we know $\mathcal{F}(X)$, $\mathcal{F}(Y)$ can also be determined.

Let $X$ be a finite set, and $\Gamma$ a group acting transitively on $X$. If $\Gamma$ is imprimitive on $X$, then it preserves a nontrivial partition of $X$, called a block system, each element of which is called a block. Clearly, if $\Gamma$ is both transitive and imprimitive, there must be a subset $B\subseteq X$ such that $1<|B|<|X|$ and $\gamma(B)\cap B=B$ or $\emptyset$ for every $\gamma \in \Gamma$. In this case, $B$ is called an imprimitive set in $X$. Furthermore, a subset $B\subseteq X$ is said to be $semi$-$imprimitive$ if $1<|B|<|X|$ and for each $\gamma\in \Gamma$ we have $\gamma(B)\cap B=B$, $\emptyset$ or $\{b\}$ for some $b\in B$.

The following theorem (cf. \cite[Theorem~1.12]{J85}) is a classical result on the primitivity of group actions.
\begin{thm}\emph{(\cite{J85})}\label{primitivity}
Suppose that a group $\Gamma$ transitively acts on $X$. Then $\Gamma$ is primitive on $X$ if and only if for each $a\in X$, $\Gamma_a$ is a maximal subgroup of $\Gamma$. Here $\Gamma_a=\{\gamma\in\Gamma:\gamma(a)=a\}$, the stabilizer of $a\in X$.
\end{thm}

Noticing the similarities about families that are cross-t-intersecting or cross-Sperner, Wang and Zhang \cite{WZ13} proved the following theorem about $\alpha(G(X,Y))$ and $I(X,Y)$ of a special kind of part-transitive bipartite graphs.

\begin{thm}\emph{(\cite{WZ13})}\label{key00}
Let $G(X,Y)$ be a non-complete bipartite graph with $|X|\leq|Y|$. If $G(X,Y)$ is part-transitive and every fragment of $G(X,Y)$ is primitive under the action of a group $\Gamma$. Then $\alpha(X,Y)=|Y|-d(X)+1$. Moreover,
\begin{itemize}
{\item[(1)] If $|X|<|Y|$, then $X$ has only 1-fragments;}
{\item[(2)] If $|X|=|Y|$, then each fragment in $X$ has size 1 or $|X|-d(X)$ unless there is a semi-imprimitive fragment in $X$ or $Y$.}
\end{itemize}
\end{thm}

To deal with multi-part cross-intersecting families, we introduce the following variation of Theorem \ref{key00}.

\begin{thm}\label{key01}
Let $G(X,Y)$ be a non-complete bipartite graph with $|X|\leq|Y|$. If $G(X,Y)$ is part-transitive under the action of a group $\Gamma$. Then
\begin{equation}\label{eq04}
\alpha(X,Y)=\max{\{|Y|-d(X)+1, |A'|+|Y|-|N(A')|, |B'|+|X|-|N(B')|\}},
\end{equation}
where $A'$ and $B'$ are minimum imprimitive subsets of $X$ and $Y$ respectively. By minimum, here we mean that $|N(A')|-|A'|=\min{\{|N(A)|-|A|: A\in X~(\text{or~} Y) \text{~is~imprimitive}\}}$.
\end{thm}

For the proof of Theorem~\ref{key01}, we need the following two lemmas from \cite{WZ13}.

\begin{lem}\emph{(\cite{WZ13})}\label{lem01}
Let $G(X,Y)$ be a non-complete bipartite graph. Then, $|Y|-\epsilon(X)=|X|-\epsilon(Y)$, and
\begin{itemize}
{\rm\item[(i)] $A\in \mathcal{F}(X)$ if and only if $(Y\setminus N(A))\in \mathcal{F}(Y)$ and $N(Y\setminus N(A))=X\setminus A$;}
{\rm\item[(ii)] $A\cap B$ and $A\cup B$ are both in $\mathcal{F}(X)$ if $A$, $B\in \mathcal{F}(X)$, $A\cap B\neq \emptyset$ and $N(A\cup B)\neq Y$.}
\end{itemize}
\end{lem}

\begin{lem}\emph{(\cite{WZ13})}\label{lem02}
Let $G(X,Y)$ be a non-complete and part-transitive bipartite graph under the action of a group $\Gamma$. Suppose that $A\in\mathcal{F}(X,Y)$ such that $\emptyset\neq\gamma(A)\cap A\neq A$ for some $\gamma\in\Gamma$. Define $\phi: \mathcal{F}(X,Y)\rightarrow\mathcal{F}(X,Y)$,
\begin{flalign*}
\phi (A)=\left\{\begin{array}{ll}Y\setminus N(A),&~~~\text{if}~A\in\mathcal{F}(X);\\X\setminus N(A),&~~~\text{if}~A\in\mathcal{F}(Y).\end{array}\right.
\end{flalign*}
If $|A|\leq|\phi(A)|$, then $A\cup\gamma(A)$ and $A\cap\gamma(A)$ are both in $\mathcal{F}(X,Y)$.
\end{lem}

\begin{rem}\label{balanced}
As a direct consequence of Lemma~\ref{lem01}, a maximum-sized nontrivial independent set in $G(X,Y)$ is of the form $A\sqcup(Y\setminus N(A))$ for some $A\in \mathcal{F}(X)$ or $B\sqcup(X\setminus N(B))$ for some $B\in \mathcal{F}(Y)$. Therefore, in order to address our problems, it suffices to determine $\mathcal{F}(X)$ $(\text{or~}\mathcal{F}(Y))$.

Meanwhile, for the mapping $\phi$ in Lemma~\ref{lem02}, we have $\phi^{-1}=\phi$ and $|A|+|\phi(A)|=\alpha(X,Y)$. When $|A|=|\phi(A)|$, we call the fragment $A$ balanced. Thus, all balanced fragments have size $\frac{1}{2}\alpha(X,Y)$.
\end{rem}

\begin{proof}[\bf{Proof of Theorem~\ref{key01}}]
The same as the original proof of Theorem \ref{key00} in \cite{WZ13}, we apply Lemma~\ref{lem02} repeatedly. For any $A_0\in\mathcal{F}(X,Y)$ satisfying $|A_0|\leq|\phi(A_0)|$, if there exists $\gamma\in\Gamma$ such that $\emptyset\neq\gamma(A_0)\cap A_0\neq A_0$, then by Lemma~\ref{lem02} we have: (1) $A_0\cap\gamma(A_0)\in\mathcal{F}(X,Y)$ or (2) $\gamma(A_0)\cap A_0=\emptyset$ or $\gamma(A_0)\cap A_0=A_0$ for any $\gamma\in\Gamma$.

For case (1), denote
\begin{flalign*}
A_1=\left\{\begin{array}{ll}A_0\cap\gamma(A_0), &~\text{if}~|A_0\cap\gamma(A_0)|\leq|\phi(A_0\cap\gamma(A_0))|;
\\ \phi(A_0\cap\gamma(A_0)), &~\text{~~~~~~~~~~~~~~~otherwise};\end{array}\right.
\end{flalign*}
and consider the primitivity of $A_1$, i.e., whether there is a $\gamma'\in\Gamma$ such that $\emptyset\neq\gamma'(A_1)\cap A_1\neq A_1$ or not.

For case (2), if $|A_0|\neq1$,  according to the definition, $A_0$ is an imprimitive set of $X$ (or $Y$). Otherwise, $|A_0|=1$, which means $\mathcal{F}(X,Y)$ contains a singleton.

By doing these procedures repeatedly, after $r$ $(0\leq r\leq |A_0|-1)$ steps, we have a fragment $A_r\in\mathcal{F}(X,Y)$ such that $A_r$ is either a singleton or an imprimitive set. Hence, we have
$$\alpha(X,Y)=\max{\{|Y|-d(X)+1, |X|-d(Y)+1, |A'|+|Y|-|N(A')|, |B'|+|X|-|N(B')|\}},$$
where $A'$ and $B'$ are minimum imprimitive subsets of $X$ and $Y$ respectively. Noticing that $|Y|\geq|X|$ and $d(X)|X|=d(Y)|Y|$, we have  $d(X)=d(Y)|Y|/|X|\geq d(Y)$. Therefore,
$$|Y|-|X|=d(X)|X|/d(Y)-|X|=(d(X)-d(Y))|X|/d(Y)\geq d(X)-d(Y),$$
which implies that $|X|-d(Y)+1\leq|Y|-d(X)+1$. Finally we have
$$\alpha(X,Y)=\max{\{|Y|-d(X)+1, |A'|+|Y|-|N(A')|, |B'|+|X|-|N(B')|\}}.$$
\end{proof}

\section{Proof of Theorem~\ref{main0}}

Throughout this section, for any nonempty subset $S\subseteq[p]$ and $A=\prod_{i\in S}A_i\in\prod_{i\in S}{[n_i]\choose k_i}$, denote $\bar{A}=\prod_{i\in S}\bar{A}_i$. Before we start the proof of Theorem~\ref{main0}, we introduce the following proposition about the direct product of Kneser graphs.

\begin{prop}\label{note1}
Given a positive integer $p$, let $n_1,n_2,\ldots,n_p$ and $k_1,k_2,\ldots,k_p$ be positive integers with $n_i\ge 2k_i$ for $1\leq i\leq p$. Let $G=\prod_{i\in[p]}{KG_{n_i,k_i}}$. Then $G$ is IS-imprimitive if and only if there exists an $i\in[p]$ such that $n_i=2k_i\ge 4$ or there exist distinct $i,j\in[p]$ such that $n_i=n_j=2$ and $k_i=k_j=1$.
\end{prop}

\begin{proof}
Note that if the Kneser graph $KG_{n,k}$ is disconnected, then $n=2k\geq 4$ and $KG_{n,k}$ is bipartite. Thus by Proposition \ref{imprimitive Kneser graph}, $KG_{2k,k}$ is IS-imprimitive for all $k\geq 2$. Moreover, since $\chi({KG_{n,k}})=n-2k+2$ for all $n\geq 2k$ (Lov\'{a}sz-Kneser Theorem, see \cite{Lovasz78}), we know that if $KG_{n,k}$ is bipartite, then $n=2k\geq 2$. Now we use induction on the number of factors $p$.

If $p=2$, let $G_1=KG_{n_1,k_1}$, $G_2=KG_{n_2,k_2}$, and $G=G_1\times G_2$. W.l.o.g., assume that $\frac{\alpha(G_1)}{|G_1|}\ge\frac{\alpha(G_2)}{|G_2|}$. Then, by Theorem \ref{tensor product}, (i) $G_1\times G_2$ is MIS-normal, or (ii) $\frac{\alpha(G_1)}{|G_1|}=\frac{\alpha(G_2)}{|G_2|}$ and one of $G_1$ and $G_2$ is IS-imprimitive, or (iii) $\frac{\alpha(G_1)}{|G_1|}>\frac{\alpha(G_2)}{|G_2|}$ and $G_2$ is disconnected. For case (i), by Theorem \ref{normal-imprimitive}, at least one factor of $G$ is IS-imprimitive or both $G_1$ and $G_2$ are bipartite. Noticed that $KG_{2,1}$ is IS-primitive, therefore, either there exists an $i\in [2]$ such that $n_i=2k_i\ge 4$ or there exist distinct $i,j\in[2]$ such that $n_i=n_j= 2k_i=2k_j=2$.  For cases (ii) and (iii), since $G$ is not MIS-normal, by Theorem \ref{imprimitive product kneser graph}, at least one of $G_1$ and $G_2$ is IS-imprimitive. Thus the proposition holds when $p=2$.

Suppose the proposition holds when the number of factors is $p-1$. Set $G'_1=\prod_{i=1}^{p-1}KG_{n_i,k_i}$ and $G'_2=KG_{n_p,k_p}$, by Theorem \ref{normal-imprimitive}, at least one factor of $G'_1$ and $G'_2$ is IS-imprimitive or both $G'_1$ and $G'_2$ are bipartite. If $G_1'$ is IS-imprimitive, by the induction hypothesis, there exists an $i'\in[p-1]$ such that $n_{i'}=2k_{i'}\ge 4$ or there exist distinct $i',j'\in[p-1]$ such that $n_{i'}=n_{j'}=2k_{i'}=2k_{j'}=2$. If $G_2'$ is IS-imprimitive, then $n_p=2k_p\geq 4$. Otherwise, both $G'_1$ and $G'_2$ are IS-primitive and bipartite. Thus, for $G'_2$, we have $n_p=2k_p=2$. For $G'_1$, since $\chi(G'_1)\cdot\alpha(G'_1)\ge |V(G'_1)|$, we know that there exists $i'\in[p-1]$ such that $n_{i'}=2k_{i'}=2$ by Lemma \ref{tensor product}. This completes the proof.
\end{proof}
%
%

The idea of the proof for Theorem~\ref{main0} is similar to that for general connected symmetric systems in \cite{WZ11}. Since $\prod_{i=1}^{p}\text{KG}_{n_i,k_i}$ is a vertex transitive graph, by Lemma \ref{lemmaimprimitive}, we can prove the bound (\ref{bound1}). Then, through a careful analysis, we can obtain the structure of all imprimitive independent sets of this graph. This leads to the unique structures of $\mathcal{A}_1$ and $\mathcal{A}_2$ in $(\ref{formula})$.

\begin{proof}[\bf{Proof of Theorem~\ref{main0}}]
Define a graph $G$ on the vertex set $X=\prod_{s\in[p]}{[n_s]\choose k_s}$ with $A,B\in X$ forming an edge in $G$ if and only if $A\cap B=\emptyset$. Therefore, $G$ is the direct product of Kneser graphs $\text{KG}_{n_1,k_1}\times\cdots\times\text{KG}_{n_p,k_p}$.

Assume that $2\leq\frac{n_1}{k_1}\leq\frac{n_2}{k_2}\leq\ldots\leq\frac{n_p}{k_p}$, then $\frac{|G|}{\alpha(G)}=\frac{n_1}{k_1}$  by Theorem~\ref{tensor product}. Following the notations of Borg in \cite{B09,B10,BL10}, write $\mathcal{A}^\ast_i=\{A\in\mathcal{A}_i|A\cap B\ne\emptyset\text{ for any }B\in\mathcal{A}_i\}$, $\hat{\mathcal{A}}_i=\mathcal{A}_i\setminus\mathcal{A}_i^\ast$, $\mathcal{A}^\ast=\bigcup_{i=1}^m\mathcal{A}_i^{\ast}$, $\hat{\mathcal{A}}=\bigcup_{i=1}^m\hat{\mathcal{A}}_i$. Note that $\bar{N}_{G}[\mathcal{A}]=\{B\in X|A\cap B\ne\emptyset,\text{ for any }A\in\mathcal{A}\}$ for $\mathcal{A}\subseteq X$, it is easy to show that $\mathcal{A}^\ast$ is an intersecting family and $\hat{\mathcal{A}}\subseteq\bar{N}_{G}[\mathcal{A}^\ast]$. It follows that $\mathcal{A}_i\cap\mathcal{A}_j\subseteq\mathcal{A}_i^\ast\cap\mathcal{A}_j^\ast$ from the definition, therefore $\hat{\mathcal{A}}_i\cap\hat{\mathcal{A}}_j=\emptyset$ for $i\ne j$, and $|\hat{\mathcal{A}}|=\sum_{i=1}^m|\hat{\mathcal{A}}_i|$. Thus by Lemma~\ref{lemmaimprimitive} we have
\begin{align*}\label{imprimitive equation}
\sum\limits_{i=1}^m|\mathcal{A}_i|&=\sum\limits_{i=1}^m|\hat{\mathcal{A}}_i|+\sum\limits_{i=1}^m|\mathcal{A}_i^\ast|\leq|\hat{\mathcal{A}}|+m|\mathcal{A}^\ast|\leq|\bar{N}_{G}[\mathcal{A}^\ast]|+m|\mathcal{A}^\ast|\nonumber\\
&=\frac{|G|}{\alpha(G)}(\frac{\alpha(G)}{|G|}|\bar{N}_{G}[\mathcal{A}^\ast]|+|\mathcal{A}^\ast|)+(m-\frac{|G|}{\alpha(G)})|\mathcal{A}^\ast|\\
&\leq |G|+(m-\frac{|G|}{\alpha(G)})|\mathcal{A}^\ast|=|G|+(m-\frac{n_1}{k_1})|\mathcal{A}^\ast|.\nonumber
\end{align*}

If $m<\frac{n_1}{k_1}$, then $\sum_{i=1}^m|\mathcal{A}_i|\leq|G|$, and the equality implies $\mathcal{A}^\ast=\emptyset$. Thus $\mathcal{A}_i=\hat{\mathcal{A}_i}$ for every $i\in[m]$, and this yields that the graph $G$ is a disjoint union of the induced subgraph $G[\mathcal{A}_i]'s$. And by the cross-intersecting property, each $G[\mathcal{A}_i]$ is a connected component of $G$. Since $G$ is connected when $\frac{n_s}{k_s}>2$ for all $s\in[p]$ and $m\geq 2$, we know that one of $\mathcal{A}_i$ is $X$ and the rest are empty sets, as case (i).

If $m>\frac{n_1}{k_1}$, then $\sum_{i=1}^m|\mathcal{A}_i|\leq m\alpha(G)$, and the equality implies that $\mathcal{A}_1^\ast=\ldots=\mathcal{A}_m^\ast=\mathcal{A}^\ast$, $|\mathcal{A}^\ast|=\alpha(G)$, as case (ii).

If $m=\frac{n_1}{k_1}$, then $\sum_{i=1}^m|\mathcal{A}_i|\leq|X|$, and the equality implies that $\mathcal{A}_1^\ast=\ldots=\mathcal{A}_m^\ast=\mathcal{A}^{\ast}$ and $\frac{\alpha(G)}{|G|}|\bar{N}_{G}[\mathcal{A}^\ast]|+|\mathcal{A}^\ast|=\alpha(G)$. By Lemma \ref{lemmaimprimitive}, we know that $|\mathcal{A}^\ast|=0$, or $|\mathcal{A}^\ast|=\alpha(G)$, or $\mathcal{A}^\ast$ is an imprimitive independent set of $G$. In the last case, $\hat{\mathcal{A}}_1,\ldots,\hat{\mathcal{A}}_m$ are cross-intersecting families and form a partition of $\bar{N}_{G}[\mathcal{A}^\ast]$. In order to determine the structures of the maximum-sized cross-intersecting families in this case, we shall characterize the imprimitive independent set of $G$.

\begin{claim}\label{imprimitive independent set}
Let $\mathcal{F}=\prod_{s\in S}{[n_s]\choose k_s}$ and $X'=\prod_{s\in[p]\setminus S}{[n_s]\choose k_s}$, where $S=\{s\in [p]:~\frac{n_s}{k_s}=2\}$. If $\mathcal{A}^\ast$ is an imprimitive independent set of $G$, then $\mathcal{A}^\ast=\mathcal{A}\times X'$, where $\mathcal{A}\subseteq\mathcal{F}$ is a non-maximum intersecting family.
\end{claim}

According to Proposition \ref{note1}, $G$ is IS-imprimitive if and only if there exists an $i\in S$ such that $n_i=2k_i\ge 4$ or there exist distinct $i,j\in S$ such that $n_i=n_j=2$ and $k_i=k_j=1$. Thus, with the assumptions in this claim, $S\neq \emptyset$ and $S=\{i_0\}$ if and only if $n_{i_0}=2k_{i_0}\geq 4$ for some $i_0\in[p]$. W.l.o.g., assume that $S=[s_0]$, where $s_0=|S|$. Under this circumstance, $m=\frac{n_1}{k_1}=2$.

Divide $\mathcal{A}^\ast$ into $u$ disjoint parts $\{C_i\times \mathcal{D}_i\}_{i=1}^{u}$, where $C_i=C_{i,1}\times\ldots\times C_{i,s_0}\in\mathcal{F}$, $\mathcal{D}_i\subseteq X'$ for all $i\in [u]$ and $C_i\ne C_j$ for any $i\ne j\in[u]$. Since $N_G(C_i\times \mathcal{D}_i)=\bar{C_i}\times\mathcal{D}_i'$, where $\mathcal{D}'_i=\{A\in X': A\cap D_i=\emptyset\text{~for~some~} D_i\in\mathcal{D}_i\}$, we know that $N_G[C_i\times \mathcal{D}_i]\cap N_G[C_j\times \mathcal{D}_j]=\emptyset$ for all $i\ne j\in[u]$. Meanwhile, $C_i\times \mathcal{D}_i\cap N_G(C_j\times \mathcal{D}_j)=\emptyset$ for all $i\ne j\in[u]$. Otherwise, assume that there exists $T_1\times T_2\in C_i\times\mathcal{D}_i\cap N_G(C_j\times\mathcal{D}_j)$, for some $T_1\in\mathcal{F}$ and $T_2\in X'$. Thus we have $T_1\times T_2\cap C_j\times D_j=\emptyset$, for some $D_j\in\mathcal{D}_j$, which contradicts the fact that $\mathcal{A}^\ast$ is an intersecting family.


By projecting $G$ onto the last $p-s_0$ factors, we obtain a graph $G'$ with vertex set $X'$ such that $A,B\in X'$ form an edge in $G'$ if and only if $A,B$ are disjoint. Consider the cross-intersecting families $\{\mathcal{D}_i,\bar{N}_{G'}(\mathcal{D}_i)\}$ in $X'$, since $|\{\mathcal{D}_i,\bar{N}_{G'}(\mathcal{D}_i)\}|=2<\frac{n_{s_0+1}}{k_{s_0+1}}$, by case (i), we know that
\begin{flalign}
|\mathcal{D}_i|+|\bar{N}_{G'}(\mathcal{D}_i)|=|\mathcal{D}_i|+|X'|-|N_{G'}(\mathcal{D}_i)|\leq|X'|,\nonumber
\end{flalign}
thus we have $|\mathcal{D}_i|\leq|N_{G'}(\mathcal{D}_i)|$, and $|C_i\times \mathcal{D}_i|=|\mathcal{D}_i|\leq |N_{G'}(\mathcal{D}_i)|=|N_G(C_i\times \mathcal{D}_i)|$. Therefore
\begin{flalign*}
\frac{|\mathcal{A}^\ast|}{|N_G[\mathcal{A}^\ast]|}=\frac{\sum_{i\in[u]}|C_i\times\mathcal{D}_i|}{\sum_{i\in[u]}|N_G[C_i\times\mathcal{D}_i]|}\leq\frac{1}{2}=\frac{\alpha(G)}{|G|}=\frac{k_1}{n_1},
\end{flalign*}
and the equality holds if and only if for all $i\in[u]$, $\mathcal{D}_i=X'$ or $\bar{N}_{G'}(\mathcal{D}_i)=X'$. Since $\mathcal{D}_i\ne\emptyset$, we have $\mathcal{A}^\ast=\bigsqcup_{i=1}^{u}C_i\times X'=\mathcal{A}\times X'$. Recall that $\frac{n_s}{k_s}>2$ for all $s>s_0$, hence $C_i\cap C_j\neq \emptyset$ for any $i\neq j\in[u]$. Therefore, by the imprimitivity of $\mathcal{A}^\ast$, $\mathcal{A}^\ast$ is a non-maximum independent set of $G$, thus $\mathcal{A}\subseteq\mathcal{F}$ is a non-maximal intersecting family and the claim holds.

For every intersecting family $\mathcal{A}\subseteq\mathcal{F}$, since $\frac{n_s}{k_s}=2$ for all $s\in {S}$, then $\mathcal{A}=\{A_1,A_2,\ldots,A_w\}\times\prod_{s\in S\setminus S'}{[n_s]\choose k_s}$ for some nonempty subset $S'\subseteq S$, where $\{A_1,\ldots,A_w\}\subseteq\prod_{s\in S'}{[n_s]\choose k_s}$ satisfying $A_i\neq \bar{A}_j$ for all $i\neq j\in [w]$. In particular, if $\mathcal{A}$ is a maximum intersecting family, we can obtain that $\bigsqcup_{j=1}^{w}\{A_j,\bar{A}_j\}=\prod_{s\in S'}{[n_s]\choose k_s}$ and $2w=\prod_{s\in S'}{n_s\choose k_s}$.

Therefore, $\mathcal{A}^\ast=\{A_1,A_2,\ldots,A_{w_0}\}\times\prod_{s\in S\setminus S_1}{[n_s]\choose k_s}\times X'$ and $N_G(\mathcal{A}^\ast)=\{\bar{A}_1,\bar{A}_2,\ldots,\bar{A}_{w_0}\}\times\prod_{s\in S\setminus S_1}{[n_s]\choose k_s}\times X'$, for some positive integer $w_0<\frac{\prod_{s\in S_1}{n_s\choose k_s}}{2}$ and nonempty subset $S_1\subseteq S$.

From the structure of the imprimitive independent set $\mathcal{A}^\ast$, we know that
\begin{flalign*}
\bar{N}_G[\mathcal{A}^\ast]=\{E_1,\bar{E}_1,E_2,\bar{E}_2,\ldots,E_v,\bar{E}_v\}\times\prod\limits_{s\in S\setminus S_1}{[n_s]\choose k_s}\times X',
\end{flalign*}
where $\emptyset\neq\{E_1,\ldots,E_v\}\subseteq\prod_{s\in S_1}{[n_s]\choose k_s}$, and $\bigsqcup_{j=1}^{w_0}\{A_j,\bar{A}_j\}\sqcup\bigsqcup_{j=1}^v\{E_j,\bar{E}_j\}=\prod_{s\in S_1}{[n_s]\choose k_s}$.

Since $E_j\times\prod_{s\in S\setminus S_1}{[n_s]\choose k_s}\times X'$ and $\bar{E}_j\times\prod_{s\in S\setminus S_1}{[n_s]\choose k_s}\times X'$ must be contained in the same one of $\hat{\mathcal{A}_1}$, $\hat{\mathcal{A}_2}$, we have
\begin{flalign}
\hat{\mathcal{A}}_1&=(\mathcal{E}\cup\tilde{\mathcal{E}})\times \prod\limits_{s\in S\setminus S_1}{[n_s]\choose k_s}\times X',\nonumber\\
\hat{\mathcal{A}}_2&=(\mathcal{E}'\cup\tilde{\mathcal{E}}')\times \prod\limits_{s\in S\setminus S_1}{[n_s]\choose k_s}\times X',\nonumber
\end{flalign}
where $\mathcal{E}\sqcup\mathcal{E}'=\{E_1,\ldots,E_v\}$ and $\tilde{\mathcal{E}}\sqcup\tilde{\mathcal{E}}'=\{\bar{E}_1,\ldots,\bar{E}_v\}$. Here we denote $\tilde{\mathcal{E}}=\{\bar{E}_{i_1},\ldots,\bar{E}_{i_l}\}$ if $\mathcal{E}=\{E_{i_1},\ldots,E_{i_l}\}\subseteq\prod_{s\in S_1}{[n_s]\choose k_s}$, for some subset $\{i_1,\ldots,i_l\}\subseteq[v]$.

Finally, to sum up,
\begin{flalign}
\mathcal{A}_1&=\mathcal{A}^\ast\sqcup\hat{\mathcal{A}}_1=(\mathcal{A}\times X')\sqcup ((\mathcal{E}\cup\tilde{\mathcal{E}})\times \prod\limits_{s\in S\setminus S_1}{[n_s]\choose k_s}\times X')\nonumber,\\
\mathcal{A}_2&=\mathcal{A}^\ast\sqcup\hat{\mathcal{A}}_2=(\mathcal{A}\times X')\sqcup ((\mathcal{E}'\cup\tilde{\mathcal{E}}')\times \prod\limits_{s\in S\setminus S_1}{[n_s]\choose k_s}\times X')\nonumber.
\end{flalign}
\end{proof}

\section{Proof of Theorem \ref{main1}}

Throughout this section, we denote $S_n$ as the symmetric group on $[n]$ and $S_C$ as the symmetric group on $C$ for $C\subseteq [n]$. For each $i\in[p]$, let $X_i$ be a finite set, then for each family $\mathcal{A}\subseteq \prod_{i\in[p]}X_i$, we denote $\mathcal{A}|_i$ as the projection of $\mathcal{A}$ onto the $i$-th factor.

For the proof of Theorem~\ref{main1}, we need the following proposition obtained by Wang and Zhang in \cite{WZ13}. 
\begin{prop}\emph{(\cite{WZ13})}\label{fragment2}
Let $G(X,Y)$ be a non-complete bipartite graph with $|X|=|Y|$ and $\epsilon(X)=d(X)-1$, and let $\Gamma$ be a group part-transitively acting on $G(X,Y)$. If each fragment of $G(X,Y)$ is primitive and there are no $2$-fragments in $\mathcal{F}(X,Y)$, then every nontrivial fragment $A\in \mathcal{F}(X)$ (if there exists) is balanced (see Remark~\ref{balanced}), and for each $a\in A$, there is a unique nontrivial fragment $B$ such that $A\cap B=\{a\}$.
\end{prop}

The proof of Theorem~\ref{main1} is divided into two parts: Firstly, we prove the bound (\ref{eq02}). Consider a non-complete bipartite graph defined by the multi-part cross-intersecting family. Through discussions about the primitivity of group $\prod_{i=1}^{p}S_{n_i}$ and careful evaluations about $|\mathcal{A}|+|\mathcal{Y}|-|N(\mathcal{A})|$, the bound (\ref{eq02}) follows from Theorem \ref{key01}. Secondly, based on a characterization of all nontrivial fragments in this bipartite graph, we determine all the structures of $\mathcal{A}$ and $\mathcal{B}$ when the bound (\ref{eq02}) is attained.

\begin{proof}[\bf{Proof of Theorem~\ref{main1}}]

With the assumptions in the theorem, we define a bipartite graph $G(\mathcal{X},\mathcal{Y})$ with $\mathcal{X}=\prod_{i=1}^{p}{[n_i]\choose t_i}$ and $\mathcal{Y}=\prod_{i=1}^{p}{[n_i]\choose s_i}$. For $A=\prod_{i=1}^{p}{A_i}\in \mathcal{X}$ and $B=\prod_{i=1}^{p}{B_i}\in \mathcal{Y}$ ($A_i\in {[n_i]\choose t_i}$ and $B_i\in {[n_i]\choose s_i}$, for every $1\leq i\leq p$), $(A,B)$ forms an edge in $G(\mathcal{X},\mathcal{Y})$ if and only if $A\cap B=\emptyset$, i.e., $A_i\cap B_i=\emptyset$ for each $1\leq i\leq p$.

It can be easily verified that $\prod_{i=1}^{p}S_{n_i}$ acts transitively on $\mathcal{X}$ and $\mathcal{Y}$, respectively, and preserves the property of cross-intersecting. Thus we have $d(\mathcal{X})=|N(A)|$ for each $A\in \mathcal{X}$, and $d(\mathcal{Y})=|N(B)|$ for each $B\in \mathcal{Y}$. Since, for each $A=\prod_{i=1}^{p}{A_i}\in \mathcal{X}$,
\begin{equation*}
N(A)=\{B=\prod_{i=1}^{p}{B_i}\in \mathcal{Y}:~A_i\cap B_i=\emptyset\text{ for each }1\leq i\leq p\}=\prod_{i=1}^{p}{[n_i]\setminus A_i\choose s_i},
\end{equation*}
we have $d(\mathcal{X})=|N(A)|=\prod_{i=1}^{p}{{n_i-t_i}\choose s_i}$. Similarly, $d(\mathcal{Y})=|N(B)|=\prod_{i=1}^{p}{{n_i-s_i}\choose t_i}$.

By Theorem~\ref{key01}, we obtain that
\begin{equation*}
\alpha(\mathcal{X},\mathcal{Y})=\max{\{|\mathcal{Y}|-d(\mathcal{X})+1, |\mathcal{A}'|+|\mathcal{Y}|-|N(\mathcal{A}')|, |\mathcal{B}'|+|\mathcal{X}|-|N(\mathcal{B}')|\}},
\end{equation*}
where $\mathcal{A}'$ and $\mathcal{B}'$ are minimum imprimitive subsets of $\mathcal{X}$ and $\mathcal{Y}$ respectively. Therefore, in order to estimate $\alpha(\mathcal{X},\mathcal{Y})$ accurately, more discussions about the sizes and the structures of the imprimitive subsets of $\mathcal{X}$ and $\mathcal{Y}$ are necessary.

\begin{claim}\label{imprimitive subset}
Let $\mathcal{A}$ and $\mathcal{B}$ be imprimitive subsets of $\mathcal{X}$ and $\mathcal{Y}$ respectively, then
\begin{align*}
\mathcal{A}&=\prod_{i\in T_1}{\{A_i,\bar{A_i}\}}\times\prod_{i\in T_2}{\{A_i\}}\times\prod_{i\in [p]\setminus(T_1\sqcup T_2)}{[n_i]\choose t_i},~\text{for some disjoint $T_1, T_2\subseteq [p]$},\\
\mathcal{B}&=\prod_{i\in R_1}{\{B_i,\bar{B_i}\}}\times\prod_{i\in R_2}{\{B_i\}}\times\prod_{i\in [p]\setminus(R_1\sqcup R_2)}{[n_i]\choose s_i},~\text{for some disjoint $R_1, R_2\subseteq [p]$},
\end{align*}
where $A_i\in {[n_i]\choose t_i}$, $B_i\in {[n_i]\choose s_i}$, $T_1\sqcup T_2\neq\emptyset$, $R_1\sqcup R_2\neq\emptyset$ and $T_2,R_2\neq[p]$. Furthermore, for each $i\in T_1$, $n_i=2t_i$ and for each $i\in R_1$, $n_i=2s_i$.
\end{claim}

If $\Gamma=\prod_{i=1}^{p}S_{n_i}$ is imprimitive on $\mathcal{X}$, then from the definition we know that $\Gamma$ preserves a nontrivial partition $\{\mathcal{X}_j\}_{j=1}^{L}$ of $\mathcal{X}$. By projecting $\mathcal{X}_j$ to the $i$-th factor, we can obtain that $\bigsqcup_{j=1}^{L}(\mathcal{X}_j|_{i})=\mathcal{X}|_i={[n_i]\choose t_i}$ and $\Gamma|_{i}=S_{n_i}$ preserving this partition of $[n_i]\choose t_i$.

It is well known that for each $A_i\in {[n_i]\choose t_i}$, the stabilizer of $A_i$ is isomorphic to $S_{t_i}\times S_{n_i-t_i}$, which is a maximal subgroup of $S_{n_i}$ if $2t_i\neq n_i$ (see e.g. \cite{NB06}). Then by Theorem~\ref{primitivity}, we obtain that $S_{n_i}$ is primitive on $[n_i]\choose t_i$ unless $2t_i=n_i$, which means for the factors with $2t_i\neq n_i$ the partition $\bigsqcup_{j=1}^{L}(\mathcal{X}_j|_{i})$ of ${[n_i]\choose t_i}$ must be a trivial partition. Thus for each $j\in L$, $\mathcal{X}_j|_{i}$ is either a singleton in ${[n_i]\choose t_i}$, or $\mathcal{X}_j|_{i}={[n_i]\choose t_i}$.

When $2t_i=n_i$, it can be easily verified that the only imprimitive subset of $[n_i]\choose t_i$ has the form $\{A_i,\bar{A_i}\}$. Therefore, for the factors with $2t_i=n_i$, the partition $\bigsqcup_{j=1}^{L}(\mathcal{X}_j|_{i})$ of ${[n_i]\choose t_i}$ is either a trivial partition, or each partition block has the form $\mathcal{X}_j|_{i}=\{A_{i,j},\bar{A}_{i,j}\}$ for some $A_{i,j}\in {[n_i]\choose t_i}$.

Since each imprimitive subset of $\mathcal{X}$ can be seen as a block of a nontrivial partition of $\mathcal{X}$, we have $\mathcal{A}=\mathcal{X}_j$ for some $j\in [L]$. From the analysis above, we know that $\mathcal{A}|_i=\{A_i\}$ or $\{A_i,\bar{A_i}\}$ for some $A_i\in {[n_i]\choose t_i}$, or $\mathcal{A}|_i={[n_i]\choose t_i}$. Therefore, set $T_1\subseteq [p]$ such that for all $i\in T_1$, $2t_i=n_i$ and $\mathcal{A}|_i=\{A_i,\bar{A_i}\}$ for some $A_i\in {[n_i]\choose t_i}$; set $T_2\subseteq [p]$ such that for all $i\in T_1$, $\mathcal{A}|_i$ is a singleton, finally, we have
\begin{equation*}
\mathcal{A}=\prod_{i\in T_1}{\{A_i,\bar{A_i}\}}\times\prod_{i\in T_2}{\{A_i\}}\times\prod_{i\in [p]\setminus(T_1\sqcup T_2)}{[n_i]\choose t_i}.
\end{equation*}
The proof for the imprimitive subsets of $\mathcal{Y}$ is the same as that of $\mathcal{X}$. Thus, the claim holds.

By Claim~\ref{imprimitive subset}, we know that for the imprimitive subsets $\mathcal{A}$ and $\mathcal{B}$ above
\begin{equation*}
|\mathcal{A}|=2^{|T_1|}\cdot\prod_{i\in [p]\setminus(T_1\sqcup T_2)}{n_i \choose t_i}~\text{and}~|\mathcal{B}|=2^{|R_1|}\cdot\prod_{i\in [p]\setminus(R_1\sqcup R_2)}{n_i \choose s_i}.
\end{equation*}
And since
\begin{align*}
N(\mathcal{A})&=\{B\in \mathcal{Y}:~A\cap B=\emptyset\text{ for some }A\in\mathcal{A}\}\\
&=\prod_{i\in T_1}{({A_i\choose s_i}\sqcup{\bar{A}_i\choose s_i})}\times\prod_{i\in T_2}{{[n_i]\setminus A_i}\choose s_i}\times\prod_{i\in [p]\setminus(T_1\sqcup T_2)}{[n_i]\choose s_i},\\
N(\mathcal{B})&=\{A\in \mathcal{X}:~A\cap B=\emptyset\text{ for some }B\in\mathcal{B}\}\\
&=\prod_{i\in R_1}{({B_i\choose t_i}\sqcup{\bar{B}_i\choose t_i})}\times\prod_{i\in R_2}{{[n_i]\setminus B_i}\choose t_i}\times\prod_{i\in [p]\setminus(R_1\sqcup R_2)}{[n_i]\choose t_i},\\
\end{align*}
we have
\begin{align*}
|N(\mathcal{A})|=2^{|T_1|}\cdot\prod_{i\in T_1}{{\frac{n_i}{2}}\choose s_i}\cdot\prod_{i\in T_2}{{n_i-t_i}\choose s_i}\cdot\prod_{i\in [p]\setminus(T_1\sqcup T_2)}{n_i\choose s_i},\\
|N(\mathcal{B})|=2^{|R_1|}\cdot\prod_{i\in R_1}{{\frac{n_i}{2}}\choose t_i}\cdot\prod_{i\in R_2}{{n_i-s_i}\choose t_i}\cdot\prod_{i\in [p]\setminus(R_1\sqcup R_2)}{n_i\choose t_i}.
\end{align*}
Now we can estimate quantities $|\mathcal{A}'|+|\mathcal{Y}|-|N(\mathcal{A}')|$ and $|\mathcal{B}'|+|\mathcal{X}|-|N(\mathcal{B}')|$.

\begin{claim}\label{size estamitae}
With the assumptions in the theorem, for all imprimitive subsets $\mathcal{A}\subseteq\mathcal{X}$ and $\mathcal{B}\subseteq\mathcal{Y}$, $|\mathcal{Y}|-d(\mathcal{X})+1> |\mathcal{A}|+|\mathcal{Y}|-|N(\mathcal{A})|$, and $|\mathcal{Y}|-d(\mathcal{X})+1> |\mathcal{B}|+|\mathcal{X}|-|N(\mathcal{B})|$.
\end{claim}

We prove the claim by estimating the difference directly. Denote
\begin{align*}
&~D_1=|N(\mathcal{A})|-|\mathcal{A}|-d(\mathcal{X})+1\text{~and}\\
D_2&=|\mathcal{Y}|-|\mathcal{X}|+|N(\mathcal{B})|-|\mathcal{B}|-d(\mathcal{X})+1
\end{align*}
to be the differences between $|\mathcal{Y}|-d(\mathcal{X})+1$ and, respectively, $|\mathcal{A}|+|\mathcal{Y}|-|N(\mathcal{A})|$ and $|\mathcal{B}|+|\mathcal{X}|-|N(\mathcal{B})|$.
Set $d_1=\frac{D_1}{|N(\mathcal{A})|}$, $d_2=\frac{D_2}{|\mathcal{X}|}$.
Then, we have $d_1=1-\beta_1-\beta_2+\theta$, $d_2=\delta+\eta_0\cdot(1-\eta_1-\eta_2)+\theta'$, where $\theta=|N(\mathcal{A})|^{-1}$, $\delta=\frac{|\mathcal{Y}|-|\mathcal{X}|}{|\mathcal{X}|}$, $\eta_0=\frac{|N(\mathcal{B})|}{|\mathcal{X}|}$, $\theta'=|\mathcal{X}|^{-1}$, $\beta_1=\frac{|\mathcal{A}|}{|N(\mathcal{A})|}$, $\beta_2=\frac{d(\mathcal{X})}{|N(\mathcal{A})|}$, $\eta_1=\frac{|\mathcal{B}|}{|N(\mathcal{B})|}$, and $\eta_2=\frac{d(\mathcal{X})}{|N(\mathcal{B})|}$.

Since ${n_i\choose t_i}\cdot{{n_i-t_i}\choose s_i}={n_i\choose s_i}\cdot{{n_i-s_i}\choose t_i}$ for each $i\in [p]$, we have $1/{{n_i-t_i}\choose s_i}={n_i\choose t_i}/({n_i\choose s_i}\cdot{{n_i-s_i}\choose t_i})$ for each $i\in[p]$. This yields that
\begin{align*}
&~~~~\beta_1=\prod_{i\in [p]}\frac{{n_i \choose t_i}}{{n_i\choose s_i}}\cdot\prod_{i\in T_1\sqcup T_2}\frac{1}{{{n_i-s_i}\choose t_i}},~~\beta_2=\frac{1}{2^{|T_1|}}\cdot\prod_{i\in [p]\setminus(T_1\sqcup T_2)}\prod_{j=0}^{s_i-1}(1-\frac{t_i}{n_i-j}),\\
\eta_1&=\prod_{i\in [p]}\frac{{n_i \choose s_i}}{{n_i\choose t_i}}\cdot\prod_{i\in R_1\sqcup R_2}\frac{1}{{{n_i-t_i}\choose s_i}},~~\eta_2=\prod_{i\in [p]}\frac{{n_i \choose s_i}}{{n_i\choose t_i}}\cdot\frac{1}{2^{|R_1|}}\cdot\prod_{i\in [p]\setminus(R_1\sqcup R_2)}\prod_{j=0}^{t_i-1}(1-\frac{s_i}{n_i-j}).
\end{align*}

By the assumptions, we know that $n_i\geq s_i+t_i+1\geq5$, $\prod_{i\in [p]}\frac{{n_i \choose t_i}}{{n_i\choose s_i}}\leq 1$ and ${{n_i-s_i}\choose t_i}\geq {\lceil\frac{n_i}{2}\rceil\choose t_i}\geq \frac{n_i}{2}$. Since $T_1\sqcup T_2\neq\emptyset$, $R_1\sqcup R_2\neq\emptyset$ and $T_2,R_2\neq[p]$, we can obtain
\begin{align*}
\beta_1&\leq\prod_{i\in T_1\sqcup T_2}\frac{1}{{{n_i-s_i}\choose t_i}}\leq\max\limits_{i\in (T_1\sqcup T_2)}{\{(\frac{2}{n_i+2})^{|T_1|}\cdot(\frac{2}{n_i})^{|T_2|}\}},\\
&~\beta_2\leq\frac{1}{2^{|T_1|}}\cdot\max\limits_{i\in[p]\setminus (T_1\sqcup T_2)}\{(1-\frac{4n_i-6}{n_i(n_i-1)})^{p-(|T_1|+|T_2|)}\},
\end{align*}
and
\begin{align*}
\eta_1&\leq(1+\delta)\cdot\prod_{i\in R_1\sqcup R_2}\frac{1}{{{n_i-t_i}\choose s_i}}\leq(1+\delta)\cdot\max\limits_{i\in (R_1\sqcup R_2)}{\{(\frac{2}{n_i+2})^{|R_1|}\cdot(\frac{2}{n_i})^{|R_2|}\}},\\
&~~~~~~~~~\eta_2\leq(1+\delta)\cdot\frac{1}{2^{|R_1|}}\cdot\max\limits_{i\in[p]\setminus (R_1\sqcup R_2)}\{(1-\frac{4n_i-6}{n_i(n_i-1)})^{p-(|R_1|+|R_2|)}\}.
\end{align*}
This leads to
\begin{flalign*}
\beta_1+\beta_2\leq\left\{\begin{array}{ll}1-\min\limits_{i\neq j\in[p]}\{\frac{6}{n_i}-\frac{2}{n_i-1}-\frac{2}{n_j}\}, &~\text{if}~T_2\neq\emptyset;\\
\frac{1}{2}-\min\limits_{i\neq j\in[p]}\{\frac{3}{n_i}-\frac{1}{n_i-1}-\frac{2}{n_j+2}\}, &~\text{otherwise};\end{array}\right.
\end{flalign*}
and
\begin{flalign*}
\frac{\eta_1+\eta_2}{1+\delta}\leq\left\{\begin{array}{ll}1-\min\limits_{i\neq j\in[p]}\{\frac{6}{n_i}-\frac{2}{n_i-1}-\frac{2}{n_j}\}, &~\text{if}~R_2\neq\emptyset;\\
\frac{1}{2}-\min\limits_{i\neq j\in[p]}\{\frac{3}{n_i}-\frac{1}{n_i-1}-\frac{2}{n_j+2}\}, &~\text{otherwise}.\end{array}\right.
\end{flalign*}
Since $5\leq n_i\leq\frac{7}{4} n_j$ for all distinct $i,j\in[p]$, thus we have $\beta_1+\beta_2,\frac{\eta_1+\eta_2}{1+\delta}\leq 1$. Therefore,
\begin{align*}
&~~~~~~~~~~~~~~~~~~~~~~~~~d_1=1-\beta_1-\beta_2+\theta> 1-\beta_1-\beta_2\geq0,\\
d_2&=\delta+\eta_0\cdot(1-\eta_1-\eta_2)+\theta'=\delta\cdot(1-\eta_0\cdot\frac{\eta_1+\eta_2}{1+\delta})+\eta_0\cdot(1-\frac{\eta_1+\eta_2}{1+\delta})+\theta'>0.
\end{align*}
Thus, the claim holds.

For each pair of non-empty cross-intersecting families $(\mathcal{A},\mathcal{B})\in 2^{\mathcal{X}}\times2^{\mathcal{Y}}$, $\mathcal{A}\cup \mathcal{B}$ forms a nontrivial independent set of $G(\mathcal{X},\mathcal{Y})$. Therefore, by Claim~\ref{size estamitae}, the inequality~(\ref{eq02}) holds.

To complete the proof, we need to characterize all the nontrivial fragments in $\mathcal{F}(\mathcal{X})$. As a direct consequence of Claim~\ref{size estamitae}, every fragment of $G(\mathcal{X},\mathcal{Y})$ is primitive. Hence, by Theorem~\ref{key00}, when $\prod_{i\in [p]}{n_i \choose t_i}<\prod_{i\in [p]}{n_i\choose s_i}$, $\mathcal{X}$ has only $1$-fragments.

When $\prod_{i\in [p]}{n_i \choose t_i}=\prod_{i\in [p]}{n_i\choose s_i}$, suppose there are nontrivial fragments in $\mathcal{F}(\mathcal{X})$. W.l.o.g., assume that $\mathcal{S}$ is a minimal-sized nontrivial fragment in $\mathcal{X}$. By Theorem~\ref{key00}, $\mathcal{S}$ is semi-imprimitive. Since for any two different elements $A,B\in\mathcal{X}$, $|N(A)\cap N(B)|<\prod_{i\in [p]}{{n_i-t_i}\choose s_i}-1$. Therefore, there are no $2$-fragments in $\mathcal{F(\mathcal{X})}$. By Proposition~\ref{fragment2}, $\mathcal{S}$ is balanced.

Now we are going to prove the non-existence of such $\mathcal{S}$ by analyzing its size and structure, which will yield that $\mathcal{X}$ also has only $1$-fragments when $\prod_{i\in [p]}{n_i \choose t_i}=\prod_{i\in [p]}{n_i\choose s_i}$.

For each $A=\prod_{i\in [p]}A_i\in \mathcal{S}$, let $\Gamma_A=\prod_{i\in [p]}(S_{A_i}\times S_{\bar{A}_i})$, $\Gamma_{\mathcal{S}}=\{\sigma\in \Gamma:~\sigma(\mathcal{S})=\mathcal{S}\}$ and $\Gamma_{A,\mathcal{S}}=\{\sigma\in \Gamma_A:~\sigma(\mathcal{S})=\mathcal{S}\}$. We claim that there exists a subset $C\in \mathcal{S}$ such that $\Gamma_{C}\neq \Gamma_{C,\mathcal{S}}$. Otherwise, for any two different subsets $B,B'\in \mathcal{S}$, we have $\Gamma_{B}=\Gamma_{B,\mathcal{S}}$ and $\Gamma_{B'}= \Gamma_{B',\mathcal{S}}$. Since $\Gamma_{B,\mathcal{S}}$ and $\Gamma_{B',\mathcal{S}}$ are both subgroups of $\Gamma_{\mathcal{S}}$, we have $\langle\Gamma_{B},\Gamma_{B'}\rangle$ is a subgroup of $\Gamma_{\mathcal{S}}$. Let $T\subseteq [p]$ be the factors where $B'_i=B_i~(\text{or}~\bar{B}_i~\text{if}~2t_i=n_i)$, write
$$\Gamma_B=\prod_{i\in T}(S_{B_i}\times S_{\bar{B}_i})\times\prod_{i\in [p]\setminus T}(S_{B_i}\times S_{\bar{B}_i}),$$
then we have,
$$\Gamma_{B'}=\prod_{i\in T}(S_{B_i}\times S_{\bar{B}_i})\times\prod_{i\in [p]\setminus T}(S_{B'_i}\times S_{\bar{B}'_i}).$$
Since $\langle S_{B_i}\times S_{\bar{B}_i}, S_{B'_i}\times S_{\bar{B}_i'}\rangle=S_{n_i}$ for each $B'_i\ne B_i~(\text{and}~B'_i\ne\bar{B}_i~\text{if}~2t_i=n_i)$, we have $$\langle\Gamma_{B},\Gamma_{B'}\rangle=\prod_{i\in T}(S_{B_i}\times S_{\bar{B}_i})\times\prod_{i\in [p]\setminus T}S_{n_i}.$$
Therefore, for some fixed $B\in\mathcal{S}$, $\Gamma_{\mathcal{S}}$ contains $\prod_{i\in T'}(S_{B_i}\times S_{\bar{B}_i})\times\prod_{i\in [p]\setminus T'}S_{n_i}$ as a subgroup, where
$$T'=\{i|i\in[p],\text{ such that }A_i=B_i~(\text{or }\bar{B}_i\text{ if }2t_i=n_i)\text{ for all }A\in\mathcal{S}\}.$$
When $T'=\emptyset$, we have $\Gamma_{\mathcal{S}}=\prod_{i\in [p]}S_{n_i}$, thus $\mathcal{S}=\mathcal{X}$, yielding a contradiction. When $T'\neq \emptyset$, if $|T'|=1$, w.l.o.g., taking $T'=\{1\}$, we have $(S_{B_1}\times S_{\bar{B}_1})\times\prod_{i\in [p]\setminus \{1\}}S_{n_i}\subseteq\Gamma_{\mathcal{S}}$. Therefore, since $\mathcal{S}\neq\mathcal{X}$, from the definition of $T'$ we have
$$\mathcal{S}=\{B_1\}\times\prod_{i\in [p]\setminus \{1\}}{[n_i]\choose t_i}, \text{~or~} S=\{B_1,\bar{B}_1\}\times\prod_{i\in [p]\setminus \{1\}}{[n_i]\choose t_i}\text{~when $2t_1=n_1$}.$$
In both cases, $|\mathcal{S}|<\frac{\alpha(\mathcal{X},\mathcal{Y})}{2}$. If $|T'|\geq 2$, we have
$$\mathcal{S}\subseteq\{B_{i_0}\}\times\prod_{i\in [p]\setminus \{i_0\}}{[n_i]\choose t_i}, \text{~or~} S\subseteq\{B_{i_0},\bar{B}_{i_0}\}\times\prod_{i\in [p]\setminus \{i_0\}}{[n_i]\choose t_i}\text{~when $2t_{i_0}=n_{i_0}$},$$
for some $i_0\in T'$. Therefore, when $T'\neq \emptyset$, we always have $|\mathcal{S}|<\frac{\alpha(\mathcal{X},\mathcal{Y})}{2}$, which contradicts the fact that $\mathcal{S}$ is balanced. Hence, the existence of $C$ is guaranteed.

By Proposition~\ref{fragment2} we have that $[\Gamma_{C}:\Gamma_{C,\mathcal{S}}]$, the index of $\Gamma_{C,\mathcal{S}}$ in $\Gamma_{C}$, equals 2. Now let $\Gamma_{C,\mathcal{S}}[C_i]$ be the projection of $\Gamma_{C,\mathcal{S}}$ onto $S_{C_i}$, $\Gamma_{C,\mathcal{S}}[C_i]$ must be a subgroup of $S_{C_i}$ of index no greater than 2. Thus $\Gamma_{C,\mathcal{S}}[C_i]=S_{C_i}$ or $A_{C_i}$. Since $\Gamma_C=\prod_{i\in [p]}(S_{C_i}\times S_{\bar{C}_i})$, we know that $\Gamma_{C,\mathcal{S}}=\prod_{i\in [p]\setminus\{j\}}(S_{C_i}\times S_{\bar{C}_i})\times(A_{C_{j}}\times S_{\bar{C}_{j}})$ or $\prod_{i\in [p]\setminus\{j\}}(S_{C_i}\times S_{\bar{C}_i})\times(S_{C_{j}}\times A_{\bar{C}_{j}}$), for some $j\in [p]$.

Since for all $i\in [p]$, $t_i=|B_i\cap C_i|+|B_i\cap \bar{C}_i|$ for each pair $B,C\in \mathcal{S}$. If $|B_i\cap C_i|>1$, let $s,t\in B_i\cap C_i$, then the transposition $(s~t)$ fixes both $C_i$ and $B_i$. Taking $i=j$, the semi-imprimitivity of $\mathcal{S}$ implies that $(s~t)\in \Gamma_{C,\mathcal{S}}|_{S_{C_{j}}\times S_{\bar{C}_{j}}}$. This yields $\Gamma_{C,\mathcal{S}}|_{S_{C_{j}}\times S_{\bar{C}_{j}}}=S_{C_{j}}\times A_{\bar{C}_{j}}$. From this process it follows that, for each $B\in \mathcal{S}$, there exists at most one of $|B_j\cap C_j|$ and $|B_j\cap \bar{C}_j|$ to be greater than $1$. Note that if $B_j\in\bar{C}_j$, then $S_{C_j}$ and $S_{B_j}$ fix both $C_j$ and $B_j$, i.e., $S_{C_j}\times S_{B_j}\subseteq \Gamma_{C,\mathcal{S}}|_{S_{C_{j}}\times S_{\bar{C}_{j}}}$. Since $\Gamma_{C,\mathcal{S}}|_{S_{C_{j}}\times S_{\bar{C}_{j}}}=A_{C_{j}}\times S_{\bar{C}_{j}}$ or $S_{C_{j}}\times A_{\bar{C}_{j}}$, and neither $A_{C_{j}}\times S_{\bar{C}_{j}}$ nor $S_{C_{j}}\times A_{\bar{C}_{j}}$ contains $S_{C_j}\times S_{B_j}$. Therefore, we obtain that $|B_j\cap C_j|=1$ for each $B\in \mathcal{S}$, or $|B_j\cap C_j|=t_j-1$ for each $B\in \mathcal{S}$.

We claim that for both cases, $\mathcal{S}$ can not be balanced.

Suppose $|B_j\cap C_j|=1$ for each $B\in \mathcal{S}$. W.l.o.g., assume $B_j\cap C_j=\{1\}$ for some $B\in \mathcal{S}$. From the semi-imprimitivity of $\mathcal{S}$, we know that for all $\gamma\in \Gamma,~\gamma(\mathcal{S})\cap\mathcal{S}=\emptyset, ~\mathcal{S}$ or $\{A\}$ for some $A\in \mathcal{S}$. Thus $(\gamma(\mathcal{S})\cap\mathcal{S})|_j=\emptyset,~\mathcal{S}|_j$ or $\{A_j\}$ for some $A_j\in {[n_j]\choose t_j}$. If $t_j>2$, then $|B_j\cap \bar{C}_j|\geq 2$, so $\Gamma_{C,\mathcal{S}}|_{S_{C_{j}}\times S_{\bar{C}_{j}}}=A_{C_{j}}\times S_{\bar{C}_{j}}$. On the other hand, we can find distinct $s,t\in C_j$ such that $(1~s~t)(B_j)=B_j\setminus\{1\}\cup\{s\}\in \mathcal{S}|_j$ since $(1~s~t)\in A_{C_j}$. Then $(1~s)(\mathcal{S}|_j)$ has more than one element of $\mathcal{S}|_j$, therefore $(1~s)\in \Gamma_{C,\mathcal{S}}|_{S_{C_{j}}\times S_{\bar{C}_{j}}}$. This contradiction proves that $t_j=2$. Thus $\mathcal{S}|_j=\mathcal{C}=\{A_j\in{[n_j]\choose 2}:~1\in A_j\}$. Otherwise, w.l.o.g., assume $C_j=\{1,2\}$ and there exists $B\in \mathcal{S}$ such that $B_j\cap C_j=\{2\}$. Since $\Gamma_{C,\mathcal{S}}|_{S_{C_{j}}\times S_{\bar{C}_{j}}}=A_{C_{j}}\times S_{\bar{C}_{j}}$ or $S_{C_{j}}\times A_{\bar{C}_{j}}$, we have $\mathcal{C}\subseteq \mathcal{S}|_j$ and $\mathcal{C'}=\{A_j\in{[n_j]\choose 2}:~2\in A_j\}\subseteq\mathcal{S}|_j$. Thus $\mathcal{S}|_j=\mathcal{C}\cup\mathcal{C'}$. This yields $\Gamma_{C,\mathcal{S}}|_{S_{C_{j}}\times S_{\bar{C}_{j}}}=S_{C_{j}}\times S_{\bar{C}_{j}}$, leading to a contradiction. 

Suppose now $|B_j\cap C_j|=t_j-1>1$ for each $B\in \mathcal{S}$. Similarly, we can prove that $n_j-t_j=2$, which contradicts $n_j\geq s_j+t_j+1$ and $2\leq s_j$, $t_j\leq \frac{n}{2}$. Therefore, for each $B\in \mathcal{S}$, $|B_j\cap C_j|=1$.

From the analysis above, we know that for each $B\in \mathcal{S}$, $B_j=\{1,b\}$ for some $b\in[n_j]$. Thus, for each $B\in \mathcal{S}$, we have $\Gamma_{B,\mathcal{S}}|_{S_{B_{j}}\times S_{\bar{B}_{j}}}=A_{B_{j}}\times S_{\bar{B}_{j}}$, and $\Gamma_{B,\mathcal{S}}=\prod_{i\in[p]\setminus \{j\}}(S_{B_i}\times S_{\bar{B}_i})\times(A_{B_{j}}\times S_{\bar{B}_{j}})$ since $[\Gamma_{B}:\Gamma_{B,\mathcal{S}}]=2$. Therefore $\Gamma_{\mathcal{S}}$ contains
$$\langle\Gamma_{B,\mathcal{S}},\text{~for~all~}B\in \mathcal{S}\rangle=\prod_{i\in T''}(S_{C_i}\times S_{\bar{C}_i})\times\prod_{i\in [p]\setminus (T''\cup\{j\})}S_{n_i}\times S_{[n_j]\setminus\{1\}}$$
as a subgroup, where $T''=\{i|i\in[p],\text{ such that }B_i=C_i~(\text{or }\bar{C}_i\text{ if }2t_i=n_i)\text{ for all }B\in\mathcal{S}\}$. Similarly, by arguing the structure of $\mathcal{S}$, if $T''\neq\emptyset$, we can prove that $|\mathcal{S}|<\frac{\alpha(\mathcal{X},\mathcal{Y})}{2}$. Thus we have $T''=\emptyset$ and $\mathcal{S}=\prod_{i\in[p]\setminus\{j\}}{[n_i]\choose t_i}\times \mathcal{C}$.

Since $\mathcal{S}$ is balanced, $\prod_{i\in [p]}{n_i \choose t_i}=\prod_{i\in [p]}{n_i\choose s_i}$ and $|\mathcal{S}|=\prod_{i\in[p]\setminus\{j\}}{n_i\choose t_i}\cdot (n_j-1)$, we have
\begin{equation}\label{eq05}
2\prod_{i\in[p]\setminus\{j\}}{n_i\choose t_i}\cdot (n_j-1)=\prod_{i\in[p]\setminus\{j\}}{n_i\choose t_i}\cdot{n_j\choose 2}-\prod_{i\in[p]\setminus\{j\}}{{n_i-s_i}\choose t_i}\cdot{{n_j-s_j}\choose 2}+1,
\end{equation}
which means $n_j$ must be an integral zero of the following function
\begin{align*}
H(x)=(1-a_0)\cdot x^2-(5-a_0\cdot(2s_j+1))\cdot x+(2b_0+4-a_0\cdot (s_j^2+s_j)),
\end{align*}
where $a_0=\prod_{i\in [p]\setminus \{j\}}\frac{{{n_i-s_i}\choose t_i}}{{n_i\choose t_i}}$ and $b_0=\prod_{i\in [p]\setminus \{j\}}{{n_i\choose t_i}^{-1}}$. Since $n_j\geq3+s_j$ and $2\leq s_j\leq \frac{n_j}{2}$, by Vieta's formulas for quadratic polynomials, there is no such $n_j$ satisfying $H(n_j)=0$ when $s_j\geq 3$. Hence $\mathcal{S}=\prod_{i\in[p]\setminus\{j\}}{[n_i]\choose t_i}\times \mathcal{C}$ is a nontrivial balanced fragment of $\mathcal{X}$ if and only if $t_j=s_j=2$ and equation~(\ref{eq05}) holds. Using the fact that $\frac{{{n_i-s_i}\choose t_i}}{{n_i\choose t_i}}\leq (1-\frac{s_i}{n_i})(1-\frac{s_i}{n_i-1})$ and the assumption $n_i\leq \frac{7}{4}n_j$ for distinct $i,j\in[p]$, it can be easily verified that the LHS of equation~(\ref{eq05}) is strictly less than the RHS when $s_j=2$. Therefore, $\mathcal{S}$ can not be balanced.

This completes the proof.
\end{proof}

\section{Concluding remarks}

In this paper we have investigated two multi-part generalizations of the cross-intersecting theorems. Our main contribution is determining the maximal size and the corresponding structures of the families for both trivially and nontrivially (with the non-empty restriction) cross-intersecting cases.

The method we used for the proof was originally introduced by Wang and Zhang in~\cite{WZ11}, which was further generalized to the bipartite case in~\cite{WZ13}. This method can deal with set systems, finite vector spaces and permutations uniformly. It is natural to ask whether we can extend the single-part cross-intersecting theorems for finite vector spaces and permutations to the multi-part case. It is possible for permutations when considering the case without the non-empty restriction, and we believe it is also possible for finite vector spaces. But when it comes to the case where the families are non-empty, as far as we know, there is still no result for finite vector spaces and permutations.

For single-part families $\mathcal{A}$ and $\mathcal{B}$, it is natural to define cross-t-intersecting as $|A\cap B|\geq t$ for each pair of $A\in \mathcal{A}$ and $B\in \mathcal{B}$. But for multi-part families, when defining cross-t-intersecting between two families, the simple extension of the definition for single-part case can be confusing. Therefore, a reasonable definition and related problems for multi-part cross-$t$-intersecting families are also worth considering.



\bibliographystyle{abbrv}
\bibliography{Multi-part_cross-intersecting_families}
\end{document}